\documentclass[11pt]{amsart}
\usepackage[margin=1.5in]{geometry}

\usepackage{graphicx}
\usepackage{tikz}
\usetikzlibrary{positioning,arrows}
\usepackage{etoolbox}
\usepackage{setspace} 
\usepackage{tikz}
\usepackage{mathtools}
\usepackage{tikz-cd}
\usepackage{wasysym}
\usepackage{adjustbox}
\usepackage{mathrsfs}
\usepackage{import}
\usepackage{url}

\newcommand{\abs}[1]{\lvert #1 \rvert}

\newcommand{\gen}[1]{\langle #1 \rangle}
\newcommand{\res}[2]{{\rm res}_{#1}^{#2}}
\newcommand{\tr}[2]{{\rm tr}_{#1}^{#2}}
\newcommand{\N}[2]{{\rm N}_{#1}^{#2}}
\newcommand{\TI}[2]{\mathcal{TI}_{#1/#2}}
\newcommand{\tamb}[1]{\underline{#1}}
\newcommand{\sr}[1]{\mathscr{#1}}

\newcommand{\gent}[1]{(\!(#1)\!)}

\newcommand{\RR}{\mathbb{R}}
\newcommand{\NN}{\mathbb{N}}
\newcommand{\QQ}{\mathbb{Q}}
\newcommand{\ZZ}{\mathbb{Z}}
\newcommand{\CC}{\mathbb{C}}
\newcommand{\FF}{\mathbb{F}}

\newcommand{\XX}{\XBox}
\newcommand{\Dr}{\mathcal{D}}

\newcommand{\card}{{\rm card}}
\newcommand{\burn}{\underline{A}}

\newcommand{\Gal}{\operatorname{Gal}}

\usepackage{amssymb}
\usepackage{amsmath}
\usepackage{amsthm}
\usepackage{amsfonts}
\usepackage{amsxtra}
\usepackage[all,2cell]{xy}
\usepackage{verbatim}
\usepackage{color}
\usepackage{xcolor}
\usepackage[unicode=true, pdfusetitle,
 bookmarks=true,bookmarksnumbered=false,
 breaklinks=false,
 backref=false,
 colorlinks=true,
 linkcolor=black,
 citecolor=black,
 urlcolor=black,
 final
]{hyperref}

\usepackage{bookmark}

\theoremstyle{plain}
\newtheorem{theorem}{Theorem}[section]
\makeatletter\let\c@theorem\c@theorem\makeatother
\newtheorem{lemma}{Lemma}[section]
\makeatletter\let\c@lemma\c@theorem\makeatother
\newtheorem{corollary}{Corollary}[section]
\makeatletter\let\c@corollary\c@theorem\makeatother
\newtheorem{proposition}{Proposition}[section]
\makeatletter\let\c@proposition\c@theorem\makeatother

\makeatletter\let\c@conjecture\c@theorem\makeatother

\theoremstyle{definition}
\newtheorem{definition}{Definition}[section]
\makeatletter\let\c@definition\c@theorem\makeatother

\makeatletter\let\c@example\c@theorem\makeatother

\makeatletter\let\c@xca\c@theorem\makeatother

\theoremstyle{remark}
\newtheorem{remark}{Remark}[section]
\makeatletter\let\c@remark\c@theorem\makeatother

\makeatletter
\let\c@equation\c@theorem
\numberwithin{equation}{section}
\makeatother

\theoremstyle{plain}
\newcommand{\thistheoremname}{}
\newtheorem*{genericthm*}{\thistheoremname}

 \newcommand\nofootnote[1]{%
  \begingroup
  \renewcommand\thefootnote{}\footnote{#1}%
  \addtocounter{footnote}{-1}%
  \endgroup
}

\usepackage[capitalise]{cleveref}

\newcommand{\newrefformat}[2]{}

\crefname{lemma}{Lemma}{Lemmas}
\crefname{theorem}{Theorem}{Theorems}
\crefname{definition}{Definition}{Definitions}
\crefname{proposition}{Proposition}{Propositions}
\crefname{remark}{Remark}{Remarks}
\crefname{corollary}{Corollary}{Corollaries}
\crefname{equation}{Equation}{Equations}
\crefname{ex}{Example}{Examples}
\crefname{appsec}{Appendix}{Appendices}

\usepackage{amssymb, amsmath}
\usepackage{palatino, hyperref}
\usepackage[normalem]{ulem}
\usepackage[T1]{fontenc}

\usepackage{mathrsfs}
\usetikzlibrary{positioning,arrows}
\usepackage{fullpage}
\allowdisplaybreaks

\def\<{{\langle}}
\def\>{{\rangle}}

\begin{document}

\title[The Trace Ideal for Cyclic Extensions] {The Tambara Structure of the Trace Ideal for Cyclic Extensions}
\author[M. Calle and S. Ginnett]{Maxine Calle and Sam Ginnett\\with an Appendix by Harry Chen and Xinling Chen}

\begin{abstract}
This paper explores the Tambara functor structure of the trace ideal of a Galois extension. In the case of a (pro-)cyclic extension, we are able to explicitly determine the generators of the ideal. Furthermore, we show that the absolute trace ideal of a cyclic group is strongly principal when viewed as an ideal of the Burnside Tambara Functor. Applying our results, we calculate the trace ideal for extensions of finite fields. The appendix determines a formula for the norm of a quadratic form over an arbitrary finite extension of a finite field.
\end{abstract}
\maketitle
\nofootnote{DOI: 10.1016/j.jalgebra.2020.04.036.}
\nofootnote{\copyright 2020. This manuscript version is made available under the CC-BY-NC-ND 4.0 license \url{http://creativecommons.org/licenses/by-nc-nd/4.0/}.}

\section{Introduction and Background}\label{sec:intro}

Let $K/F$ be a Galois extension of fields with characteristic different from $2$ with (profinite) Galois group $G := \Gal(K/F)$.  For a field $L$, let $GW(L)$ denote the Grothendieck-Witt ring of (formal differences of) isometry classes of regular quadratic forms over $L$. The Scharlau transfer (with respect to the field trace ${\rm tr}_{K/F}:K\to F$) is the homomorphism of Abelian groups $GW(K)\to GW(F)$ taking $q$ to ${\rm tr}_{K/F}\circ q$.  When considered along with the restriction (\emph{i.e.}, extension of scalars) homomorphism $GW(F)\to GW(K)$, the Grothendieck-Witt ring gains the structure of a Mackey functor in the language of A.~Dress \cite{dress:1971}.

Tambara functors \cite{tambara:1993} are elaborations of Mackey functors with multiplicative norm maps in addition to restrictions and transfers.  In \cite{bachmann:2016}, T.~Bachmann shows that the (\emph{a priori} different but ultimately equal) norm maps of D.~Ferrand \cite{ferrand:1998} and M.~Rost \cite{rost:2003} turn $GW$ into a Tambara functor.  Our aim in this work is to leverage this additional structure in order to study Dress's trace homomorphism between Burnside and Grothendieck-Witt rings, the construction of which we sketch presently.

Recall that the Burnside ring $A(G)$ of $G$ is defined as the Grothendieck construction applied to the semi-ring of isomorphism classes of finite $G$-sets under disjoint union and Cartesian product.  In \cite{dress:1971}, Dress shows that the assignment $A(G)\to GW(F)$ determined by $G/H\mapsto (x\in K^H \mapsto {\rm tr}_{K^H/F}(x^2))$ is a ring map which we call the \emph{trace homomorphism}.  The trace homomorphism is surjective precisely when $K$ contains square roots of all the elements of $F$ \cite[Appendix B, Theorem 3.1]{dress:1971}, and its kernel is the \textit{trace ideal} of $K/F$.\footnote{Instead of considering the trace ideal of a fixed extension, we might first fix a group $G$ and study the intersection of trace ideals across all Galois extensions with this Galois group, as has been done in \cite{epkenhans:1998}. We will refer to this object as the \textit{absolute trace ideal} to avoid confusion, although previous literature does not include the descriptive. Epkenhans \cite{epkenhans:1998, epkenhans:1999} determines the absolute trace ideal of elementary Abelian $2$-groups, cyclic $2$-groups, and the quaternion and dihedral groups of order $8$. \label{footnote:abs TI}} 
Assembling the trace homomorphisms for subextensions of $K/F$, we get a map of Tambara functors, whose kernel is an ideal of the Burnside Tambara functor.
It is this kernel which we will determine for cyclic Galois extensions.

In order to state our main theorem, let $C_N$ denote the cyclic group of order $N$, and for $M$ dividing $N$ let $\burn(C_N/C_M) = A(C_M)$.  For $m\mid M$, let $t_{M/m}$ denote the element of $\burn(C_N/C_M)$ corresponding to the transitive $C_M$-set $C_M/C_m$ of cardinality $M/m$.

\begin{theorem}[see \cref{thm:odd deg Cn extn}, \cref{thm:c2}, \cref{thm:c4}, \cref{thm:twopowers}, and \cref{thm:TI for Cn}]
Suppose $\Gal(K/F) = C_N$ where $N$ has prime decomposition $2^\mu p_1^{\sigma_1}\cdots p_s^{\sigma_s}$. Then $\ker(\burn_{C_N}\to \tamb{GW}_F^K)$, seen as a Tambara ideal of $\burn_{C_N}$, is generated by\begin{enumerate}
    \item $t_{p_i/1}-p_it_{p_i/p_i}$ for $i=1,\dots, s$,
    \item a generator $\sr{G}$ which is determined by $K/K^{C_{2^\mu}}$. 
    If $\mu=0$, then $\sr{G}=0$ as well. If $K/K^{C_{2^\mu}}$ is quadratic, then $\sr{G}$ is determined by the discriminant.
    Otherwise (for $\mu\geq 2)$, $\sr{G}$ depends on the discriminant of both $K/K^{C_{2}}$ and $K^{C_2}/K^{C_{4}}$, as well as
    an embedding condition on $K/K^{C_{4}}$.
\end{enumerate} 

\end{theorem}

We handle the pro-cyclic case as well in \cref{thm:p-adic trace ideal} and \cref{thm:alg closure}.  This allows a Tambara-theoretic description of the trace ideal for $\overline{\mathbb{F}_q}/\mathbb{F}_q$ in \cref{thm:finite fields}. These results permit the description of every element of the trace ideal of a (pro-)cyclic extension as the transfer of the product of the norm of the restriction of our specified generators. See \cref{rmk:lvl description of gen ideal} for a precise version of this observation.

The homotopy theory-inclined reader will note that the Burnside and Grothendieck-Witt functors appear as endomorphisms of the sphere spectrum in stable equivariant and motivic homotopy theory, respectively.  Moreover, the classical Galois correspondence induces a symmetric monoidal functor from Galois-equivariant spectra to motivic spectra \cite[Theorem 4.6]{heller/ormsby:2014}.  This functor restricts to one between highly structured normed ring spectra \cite[Proposition 10.8]{bachmann/hoyois:2018}.  In both categories, the sphere spectrum is a normed algebra, and the Dress map is the induced Tambara map between endomorphisms of the unit object.  We hope that a better understanding of this map will provide insight on the relation between equivariant and motivic homotopy theory provided by the Galois correspondence.
\subsection*{Outline}

\cref{sec:background} is dedicated to discussing Tambara functors and Tambara ideals, introducing the Burnside and Grothendieck-Witt rings as Tambara functors, and detailing the Dress map and the trace ideal for these Tambara functors. \cref{sec:A(Cn)} describes the behavior of the Burnside Tambara functor on a cyclic group, providing explicit formulas for the Tambara maps. In \cref{sec:TI(Cn)}, we calculate the trace ideal when $G$ is a cyclic group (\cref{thm:odd deg Cn extn}, \cref{thm:twopowers}, and \cref{thm:TI for Cn}). 
Building on this work, we consider the Galois groups $\mathbb{Z}_p$ and $\hat{\mathbb{Z}}$ in \cref{sec:profinite}. We show that the trace ideal for these profinite groups are colimits of the principal ideals of their finite cyclic counterparts (\cref{thm:p-adic trace ideal} and \cref{thm:alg closure}). Finally, \cref{sec:applications} applies these calculations to specific examples, including a complete description of the trace ideal for extensions of finite fields; in this case the trace ideal is strongly principal (\cref{thm:finite fields}). 
In \cref{app:HM}, H.~Chen and X.~Chen give formulas for the restriction, transfer, and norm of an arbitrary quadratic form over a finite field (\cref{thm:res and tr for finite field} and \cref{thm:main}).

\subsection*{Acknowledgements}
We extend a big thank you to Kyle Ormsby and Ang\'elica Osorno for their mentorship. This research was conducted as part of the 2019 Collaborative Mathematics Research Group (CMRG) at Reed College and generously funded by NSF grant DMS-1709302. Additional thanks goes to Jeremiah Heller and our fellow CMRG members Nick Chaiyachakorn, Nicholas Cecil, Harry Chen, and Xinling Chen for their suggestions and support. We would also like to thank the referee for their helpful comments and suggestions that improved our exposition. 

\section{Background: Tambara, Burnside, Grothendieck, Witt, and Dress}\label{sec:background}
In this section, we recall necessary background information on Tambara functors, the Burnside and Grothendieck-Witt rings, and Dress's trace homomorphism.

\subsection{Tambara functors, ideals, and generators}
First introduced as TNR functors\footnote{Transfer, Norm, Restriction.} by D. Tambara in \cite{tambara:1993}, Tambara functors are elaborations of Mackey functors which have multiplicative norm maps in addition to restrictions and transfers. These functors were originally defined only for finite groups, but have since been extended to the profinite case \cite{nakaoka:2009}.

\begin{definition}\label{defn:tamb functor}
Let $G$ be a profinite group. Let $G$Fin and Set denote the category of finite $G$-sets and the category of Sets respectively. A \textit{Tambara functor} $T$ on $G$ is a triple $(T^*, T_+, T_{\boldsymbol{\cdot}})$ where $T^*$ is a contravariant functor $G$Fin$\rightarrow$Set and $T_+, T_{\boldsymbol{\cdot}}$ are covariant functors $G$Fin$\rightarrow$ Set such that 
\begin{enumerate}
    \item $(T^*, T_+)$ is a Mackey functor on $G$,
    \item $(T^*, T_{\boldsymbol{\cdot}})$ is a semi-Mackey functor \footnote{A semi-Mackey functor is a Mackey functor in which $T(X)$ is assigned to a commutative monoid rather than an abelian group (cf. \cite{NAKAOKA2012}).} on $G$, and
    \item given an exponential diagram
    \begin{center}
    \begin{tikzcd}
    X \arrow{d}[swap]{f} & \arrow{l}[swap]{p} A & \arrow{l}[swap]{\lambda} Z \arrow{d}{\rho}\\
    Y & & B \arrow{ll}[swap]{q}
    \end{tikzcd}
    \end{center}
    in $G$Fin (in the sense of \cite{tambara:1993}), the diagram 
    \begin{center}
    \begin{tikzcd}
    T(X) \arrow{d}[swap]{T_\cdot(f)} & \arrow{l}[swap]{T_+(p)} T(A) \arrow{r}{T^*(\lambda)} &  T(Z) \arrow{d}{T_\cdot(\rho)}\\
    T(Y) & & T(B) \arrow{ll}[swap]{T_+(q)}
    \end{tikzcd}
    \end{center}
    commutes.
\end{enumerate}
For the sake of brevity, we use the notation $f_{+}:=T_{+}(f)$, $f_{\boldsymbol{\cdot}}:=T_{\boldsymbol{\cdot}}(f)$, and $f^*:=T^*(f)$.
\end{definition}

\begin{remark}
As is the case for Mackey functors, it suffices to specify how a Tambara functor behaves on transitive $G$-sets.  This observation prompts a second characterization of Tambara functors in terms of subgroups of $G$:

Let $G$ be a profinite group. A Tambara functor $T$ on $G$ is completely specified by a ring $T(G/H)$ for all open $H\leq G$ and the following maps for all open subgroups $L\leq H\leq G$:
\begin{enumerate}
    \item Restriction $\res{L}{H} := q^*$
    \item Transfer $\tr{L}{H} := q_+$
    \item Norm $\N{L}{H} := q_{\boldsymbol{\cdot}}$
    \item Conjugation $c_{g, H} := (c_g)^*$
\end{enumerate}
where $q\colon G/K\rightarrow G/H$ is the quotient map and $c_g\colon H \rightarrow H^g$ is  conjugation-by-$g$ and $H^g=g^{-1}Hg$. These maps must satisfy a number of compatibility conditions as specified M.~Hill and K.~Mazur in \cite{mazur:2019}. 
Notably, while restriction and conjugation are ring maps, transfer and norm only respect the additive and multiplicative structures, respectively.
\end{remark}

There are many analogies to draw between commutative ring theory and Tambara functor theory; the additional Tambara structure makes the study of ideals a particularly rich source of information.

\begin{definition}\label{defn:tamb ideal}
Let $T$ be a Tambara functor on $G$. An \textit{ideal} $\sr{I}$ of $T$ consists of a collection of standard ring-theoretic ideals $\sr{I}(G/H) \subseteq T(G/H)$ for each open $H\leq G$, such that for all open subgroups $L\leq H\leq G$
\begin{enumerate}
    \item $\res{L}{H}(\sr{I}(G/H)) \subseteq \sr{I}(G/L)$,
    \item $\tr{L}{H}(\sr{I}(G/L)) \subseteq \sr{I}(G/H)$,
    \item $\N{L}{H}(\sr{I}(G/L)) \subseteq \sr{I}(G/H) + \N{L}{H}(0)$,
    \item $c_{g, H}(\sr{I}(G/L)) \subseteq \sr{I}(G/{}^{g}\!H)$.
\end{enumerate}
\end{definition}

\begin{definition}\label{defn:gen tamb ideal}
Let $\mathcal{G}\subseteq \coprod_{G{\rm Fin}}T(X)$. The \textit{ideal generated by $\mathcal{G}$} is the intersection of all ideals of $T$ containing $\mathcal{G}$, denoted $\gent{\mathcal{G}}$. If $\mathcal{G}=\{a_1,\dots,a_n\}$ is finite, then we write $\gent{\mathcal{G}}=\gent{a_1,\dots,a_n}$.
\end{definition}

We will use the notation $(a)\subseteq T(X)$ to denote the ideal generated by $a\in T(X)$ in the standard ring-theoretic sense, whereas $\gent{a}\subseteq T$ denotes the \textit{Tambara} ideal generated by $a\in T(X)$.

\begin{definition}\label{defn:principal}
Let $\sr{I}$ be an ideal of $T$. We say $\sr{I}$ is \textit{principal} if $\sr{I}=\gent{a}$ for some $a\in T(X)$ where $X$ is a $G$-set. We say an ideal is \textit{strongly principal} if there exists an open subgroup $H\leq G$ and $a \in T(G/H)$ such that $\sr{I} = \gent{a}$.
\end{definition}

Nakaoka \cite[Proposition 3.6]{nakaoka:2011} shows that every finitely generated Tambara ideal is in fact principal, however the strongly principal condition is much more restrictive.

Just as ring-theoretic ideals can be produced by ring homomorphisms, Tambara ideals arise as the kernels of Tambara functor morphisms (see \cite[Proposition 2.10]{nakaoka:2011}). For more details regarding the theory of Tambara ideals, we point the reader to Nakaoka's paper \cite{nakaoka:2011}.

\begin{remark}\label{rmk:lvl description of gen ideal}
There is a particularly nice description of a Tambara ideal generated by a single element as shown in \cite[Definition 1.4]{nakaoka:2011}.
Let $T$ be a Tambara functor and let $a\in T(A)$ for some finite $G$-set $A$. Then $\gent{a}$ can be calculated as 
$$\gent{a}(X) = \{u_+(b\cdot(v_!(w^*(a))) \;|\; X\xleftarrow{u} C \xleftarrow{v} D \xrightarrow{w} A, b\in T(C)\}$$
where $v_! = v_{{\boldsymbol{\cdot}}} - v_{{\boldsymbol{\cdot}}}(0)$.
\end{remark}

\begin{definition}
Given two Tambara functors $T$ and $S$ on a profinite group $G$, a \textit{morphism of Tambara functors} $\varphi\colon T \rightarrow S$ is a collection of ring homomorphisms $\varphi_X\colon T(X)\rightarrow S(X)$ for all finite $G$-sets $X$ which form a natural transformation of each of the three component functors. The \textit{kernel} of $\varphi$ is given by the collection of kernels of the associated ring maps.
\end{definition}

\begin{remark}
Just as in standard ring theory, a surjective Tambara functor morphism $\varphi\colon T\rightarrow S$ yields the presentation $T/\ker(\varphi) \cong S$. As before, it suffices to specify $\varphi_H\colon T(G/H) \rightarrow S(G/H)$ on each open subgroup $H\leq G$. So then $\ker(\varphi)=\{\ker(\varphi_H)\}_{H\leq G}$.
\end{remark}

Our object of interest, the trace ideal, arises as the kernel of a Tambara functor morphism. The details of the morphism and the functors it maps between are detailed in the following two subsections.


\subsection{The Grothendieck-Witt and Burnside rings as Tambara functors}

The focus of our work is the \textit{Burnside Tambara functor} on $G$, denoted $\burn_G$, and the \textit{Grothendieck-Witt Tambara functor} on a field extension $K/F$ with Galois group $G$, denoted $\tamb{GW}_F^K$. This section recalls some basic definitions and results concerning these two functors. 

\begin{definition}[The Burnside ring] \label{defn:burnside ring}
The \textit{Burnside ring} on a group $G$ is the Grothendieck construction on the semi-ring of finite $G$-sets, denoted $A(G)$. That is, $A(G)$ is the ring of formal differences of isomorphism classes of finite $G$-sets, with addition given by disjoint union and multiplication given by Cartesian product.
\end{definition}

\begin{definition}[Burnside Tambara functor] \label{defn:burnside tamb functor}
For each open $H\leq G$, define $\burn_G(G/H)$ to be the Burnside ring of $H$. For open subgroups $L\leq H\leq G$, $g\in G$, $Y\in \burn_G(G/H)$, and $X\in\burn_G(G/L)$, we define the Tambara structure maps,\begin{align*}
    \res{L}{H}\colon \burn_G(G/H)&\longrightarrow \burn_G(G/L)\\
    Y &\longmapsto Y\text{ with restricted $L$-action,}\\
    \tr{L}{H}\colon \burn_G(G/L)&\longrightarrow \burn_G(G/H)\\
    X &\longmapsto H\times_L X,\\
    \N{L}{H}\colon \burn_G(G/L)&\longrightarrow \burn_G(G/H)\\
    X &\longmapsto {\rm Map}_L(H,X),\\
    c_{g,H}\colon \burn_G(G/H)&\longrightarrow \burn_G(G/H^g)\\
    X &\longmapsto X^g.
\end{align*}
When $G$ is Abelian, conjugation is trivial. These maps turn $\burn_G$ into the initial $G$-Tambara functor (cf. \cite{nakaoka:2014, tambara:1993}). We will denote $\burn_G$ by merely $\burn$ when the group is obvious from context.
\end{definition}

\begin{remark}
For any $L\leq H\leq G$, we have a natural isomorphism (cf. \cite[Remark 1.5]{nakaoka:2014}) $\burn_G(G/L)\cong \burn_H(H/L)$, and we will often identify the two through this isomorphism.
\end{remark}

We now turn to our second functor of interest, the Grothendieck-Witt functor. In order to work with this functor, we use some basic notation from the theory of quadratic forms. For a field $F$, let $F^{\times}$ be the multiplicative group of units and $F^{\XX}=\{x^2\mid x\in F^{\times}\}$ the group of squares, and so $F^{\times}/F^{\XX}$ denotes the group of square classes in $F$. The diagonal $F$-form $a_1x_1^2+\dots a_nx_n^2$ is denoted $\gen{a_1,\dots, a_n}$, with dimension $n$ and determinant $\prod_{i=1}^n a_iF^{\XX}\in F^{\times}/F^{\XX}$. For an in-depth treatment of quadratic forms, we point the reader to \cite{lam:2005}.

\begin{definition}[Grothendieck-Witt ring] \label{defn:GW ring}
Let $M(F)$ denote the set of isometry classes of regular quadratic forms over a field $F$. Equipped with the orthogonal sum and tensor product operations, $M(F)$ forms a semi-ring. The \textit{Grothendieck-Witt ring} of $F$ is the Grothendieck construction applied to this semi-ring, denoted $GW(F)$. A typical element of $GW(F)$ is a formal difference of isometry classes of quadratic forms.
\end{definition}

\begin{definition}[Grothendieck-Witt Tambara functor] \label{defn:GW tamb functor}
For each open subgroiup $H\leq G$, define $\tamb{GW}_F^K(G/H)$ to be the Grothendieck-Witt ring on $K^H$. For convienence we will often index the the Grothendieck-Witt Tambara functor by the field $K^H$ instead of the $G$-set $G/H$ under the correspondance $\tamb{GW}_F^K(G/H) = \tamb{GW}_F^K(K^H) = GW(K^H)$. 

For subextensions $F\subseteq L\subseteq E\subseteq K$, the restriction map $\res{L}{E}:\tamb{GW}_F^K(L)\longrightarrow\tamb{GW}_F^K(E)$ is given by an extension of scalars. Equipped with the Scharlau transfer of the field trace and the Rost norm \cite{rost:2003}, the Grothendieck-Witt ring naturally exhibits the structure of a Galois Tambara functor (cf. \cite[\S 3]{bachmann:2016}). Again, we denote $\tamb{GW}_F^K$ by $\tamb{GW}$ when the extension is obvious from context.
\end{definition}


\subsection{The Dress map and the trace ideal}\label{subsec:Dress and TI}

As shown in \cite{dress:1971}, there is a ring homomorphism $D_G$ between the Burnside ring on $G$ and the Grothendieck-Witt ring on $F$ which sends a transitive $G$-set $G/H$ to the trace form $\tr{F}{K^H}\gen{1} = \gen{K^H}$. Similarly, for each open $H\leq G$, we have a ring homomorphism $\Dr_H\colon \burn(G/H)\to \tamb{GW}_F^K(K^H)$. We define the \emph{Dress map} $\Dr: \burn_G \rightarrow \tamb{GW}_F^K$ as the collection of these ring homomorphisms. To see that the Dress map is indeed a Tambara functor morphism, first recall that $\burn_G$ is the initial element in the category of Tambara functors on $G$. As such, there must be a unique Tambara functor morphism $\Phi:\burn_G\rightarrow \tamb{GW}_F^K$. Now note that for all $L\leq H\leq G$ we have $H/L = \tr{L}{H}(L/L)$, and so it follows that \[\Phi(H/L) = \Phi(\tr{L}{H}(L/L)) = \tr{K^L}{K^H}(\Phi(L/L)) = \tr{K^L}{K^H}\gen{1}.\] This shows that $\Phi$ is indeed the Dress map, and hence $\Dr$ is a Tambara functor morphism. 

\begin{remark}
\label{rem:surjective}
A Tambara functor morphism is said to be surjective if the ring maps at each level are surjective. The conditions for the Dress map to be surjective at each level are given by \cite[Appendix B, Theorem 3.1]{dress:1971}.
It follows that the Dress map is a surjective Tambara functor morphism if and only if $K$ is quadratically closed. 
\end{remark}

\begin{definition}\label{defn:TI}
The \textit{trace ideal}, denoted $\TI{K}{F}$, is the Tambara ideal $\ker(\Dr\colon\burn_G\to \tamb{GW}_{F}^{K})$. When the extension is clear, we drop the decoration. 
\end{definition}

\begin{remark}\label{rmk: TI G/e}
If $G$ is a finite group, then $\TI{K}{F}(G/e) =(0)$.
\end{remark}

Note that \cref{rmk:lvl description of gen ideal} allows us to calculate the classical trace ideal just in case $\TI{K}{F}$ is principal. To gain insight into the trace ideal, we observe that for every sub-extension $F\subseteq E\subseteq K$, we have the ring homomorphism $\dim_E\colon GW(E) \rightarrow \ZZ$ which sends a quadratic form to its dimension over $E$. These maps naturally assemble into the Tambara functor morphism $\dim\colon\tamb{GW}_F^K \rightarrow \tamb{\ZZ}$ given by $\dim = \{\dim_{K^H}\}_{H \leq G}$. Here $\tamb{\ZZ}$ denotes the constant Tambara functor on $\ZZ$ given by $\tamb{\ZZ}(G/H) = \ZZ$ and 
\begin{align*}
    \res{L}{H}(a) &= a, \\
    \tr{L}{H}(a) &= |H:L|a, \\
    \N{L}{H}(a) &= a^{|H:L|}, \\
    c_{g, H}(a) &= a
\end{align*}
for all open $L \leq H \leq G$, $a\in\ZZ$, and $g\in G$. 
Similarly, for all $H\leq G$ there is a ring homomorphism $\card_H\colon A(H) \to \ZZ$ sending a finite $H$-set to its cardinality. We then have the Tambara functor morphism $\card\colon \burn_G \rightarrow \tamb{\ZZ}$ given by $\card = \{\card_H\}_{H\leq G}$. Note that $\card = \dim\circ\Dr$, which implies that $\TI{K}{F}$ is a sub-ideal of the kernel of $\card$. When the order of $G$ is odd, we have a stronger result.

\begin{theorem}\label{thm:odd deg ker card}
If the order of $G$ is odd,
\begin{center}\begin{tikzcd}
\TI{K}{F} = \ker(\card).
\end{tikzcd}\end{center}
\end{theorem}

\begin{proof}
When the order of $G$ is odd, it is well-known that $\tr{F}{K}\gen{1}=\abs{K:F}\gen{1}$ (cf. \cite{lam:2005}). For $H\leq G$, an arbitrary element of $\burn(G/H)$ is of the form $X = \sum_{L\leq H}m_LH/L$. Since every sub-extension of an odd degree extension is itself an odd degree extension, we have \begin{align*}
    \Dr_{H}(X)&=\sum_{L\leq H} m_L \tr{K^{L}}{K^{H}}\gen{1}\\
    &= \bigg(\sum_{L\leq H}m_L|H:L|\bigg)\gen{1} 
\end{align*} 
for all $H\leq G$,
implying 
$$X \in \TI{K}{F}(G/H) \iff \sum_{L\leq H}m_L|H:L| = \card(X) = 0.$$
\end{proof}
\begin{remark}\label{rmk:odd deg TI prime}
In this case, $\TI{K}{F}$ is a prime ideal of $\burn_G$ as defined in \cite{nakaoka:2011}. This observation follows from \cite[Corollary 4.29]{nakaoka:2011} and the fact that $\tamb{\ZZ}$ is a domain-like Tambara functor.
\end{remark}

With reference to \cref{footnote:abs TI}, we also define the absolute trace ideal of the Dress Tambara functor morphism. 

\begin{definition}\label{defn:abs TI}
Let $G$ be a (profinite) group. The \textit{absolute trace ideal} is the Tambara ideal given by \[
\mathcal{T}_G = \bigcap \TI{K}{F},
\] where the intersection (in the sense of \cite[\S 3.1]{nakaoka:2011}) ranges over all Galois extensions $K/F$ with Galois group $G$.
\end{definition}

\section{The Burnside Tambara Functor for a Cyclic Group}\label{sec:A(Cn)}
For $N\in \NN$, let $C_N$ denote the cyclic group of order $N$. To establish some preliminary results for the Burnside Tambara functor on $C_N$, we introduce the following notation.

\begin{definition}
Let $C_K\leq C_M\leq C_N$,  and $k=\abs{C_M:C_K}$.
Define $t_{M, k}$ to be the transitive $C_M$-set with cardinality $k$,
$$t_{M,k}:=C_M/C_K\in\burn_{C_N}(C_N/C_M).$$ We will drop the first index when it can be inferred from context.
\end{definition}
Adopting this notation, multiplication of $C_M$ sets is given by the formula
\begin{equation}\label{eqn:mult}
t_{k}t_{j} = \text{gcd}(k, j)t_{\text{lcm}(k, j)},
\end{equation}
as can be derived by considering the pullback diagram over $C_M/C_K$ and $C_M/C_J$. The element $t_{1}$ is the multiplicative identity and so will be written as $1$. 

The formulas for the restriction and transfer maps are relatively straight-forward upon consideration of the maps given in \cref{defn:burnside tamb functor}. For arbitrary elements $X=\sum_{i\mid K} a_it_{i}\in \burn(C_N/C_K)$ and $Y=\sum_{i\mid M} b_it_{i}\in \burn(C_N/C_M)$, we have\begin{align}
    \res{K}{M}(Y):=& \res{C_K}{C_M}(Y)=\sum_{i\mid M} b_id_it_{ \frac{i}{d_i}},\\
    \tr{K}{M}(X) :=& \tr{C_K}{C_M}(X) = \sum_{i\mid K} a_it_{ik},
\end{align} where $d_i={\rm gcd}(i,k)$. Conjugation is trivial. By an application of \cite[Definition 7.2 and Corollary 7.6]{nakaoka:2014}, we obtain the formula for the norm
\begin{equation}\label{eqn:norm}
    \N{K}{M}(X):=\N{C_K}{C_M}(X) = \sum_{i\mid M}\frac{C(i)}{i}t_{i},
\end{equation}
where
\[C(i) = \Bigg(\sum_{j\mid \frac{{\rm lcm}(i, k)}{k}}ja_j\Bigg)^{{\rm gcd}(i, k)} - \sum_{j \mid i,\; j < i}C(j).\]

Note that we always have $C(1)=a_1$. Using these explicit formulas, we can establish some useful lemmas regarding an arbitrary ideal $\sr{I}\subseteq \burn_{C_N}$.

\begin{lemma}\label{lem:prime-gens}
Let $C_K\leq C_M\leq C_N$ and suppose that $n(t_p-p)\in\sr{I}(C_N/C_K)$ for some odd prime $p$ and $n\in \NN$. Let $p^j$ be the largest power of $p$ dividing $\frac M K$. Then for all $0\leq i\leq j$,
$n(t_{p^{i+1}} - p^{i+1})\in\sr{I}(C_N/C_M)$. In particular we always have $n(t_p-p)\in\sr{I}(C_N/C_M)$. 
\end{lemma}

\begin{proof}
We proceed by induction on $i$. For the base case $i=0$, we wish to show that $n(t_p-p) \in \sr{I}(C_N/C_M)$. One can easily perform induction on the number of prime divisors of $\frac MK$, so it suffices to prove the base case for $\frac MK = q$ for some prime $q$. In the case $q = p$, an application of \cref{eqn:norm} gives $$\N{K}{M}(n(t_p-p)) = n^pp^{p-2}t_{p^2} -(n^pp^{p-1}-n)t_p - np.$$  Hence \begin{align*}
    \N{K}{M}(n(t_p-p)) - &n^pp^{p-2}\tr{K}{M}(n(t_p-p)) = n(t_p-p) \in \sr{I}(C_N/C_M),
\end{align*}
as desired. When $q \neq p$, applying \cref{eqn:norm} yields
$$\N{K}{M}(n(t_p-p)) = \frac{n^q(-p)^{q-1}-n}{q}t_{qp}+\frac{(-pn)^q+np}{q}t_{q} + nt_p-np.$$
Therefore
$$\N{K}{M}(n(t_p-p)) - \frac{n^q(-p)^{q-1}-n}{q}\tr{K}{M}(n(t_p-p)) = n(t_p-p) \in \sr{I}(C_N/C_M),$$
which proves the base case. Now let $1\leq i\leq j$ and assume the claim holds for all $k < i-1$. Then in particular we have that $n(t_p-p) \in \sr{I}(C_N/C_{M})$ and $n(t_{p^i} - p^i) \in \burn(C_N/C_{M/p})$. Thus we have $$\tr{{M/p}}{M}(n(t_{p^i} - p^i)) + p^in(t_p-p) = n(t_{p^{i+1}} - p^{i+1}) \in \sr{I}(C_N/C_{M})$$ for all $1\leq i \leq j$, which completes the proof.  
\end{proof}

\begin{lemma}\label{lem:2gens}
Let $C_K\leq C_M\leq C_N$ and suppose that $t_4-t_2-2\in\sr{I}(C_N/C_K)$. Let $2^j$ be the largest power of $2$ dividing $\frac M K$. Then for all $1\leq i\leq j+2$,
$t_{2^i} +t_2 - (2^{i}+2)\in\sr{I}(C_N/C_M)$. Furthermore, $2t_2-2\in \sr{I}(C_N/C_{K/2})$.
\end{lemma}

\begin{proof}
To see that $2t_2-2\in \sr{I}(C_N/C_{K/2})$, we observe that $\res{C_{K/2}}{C_K}(t_4-t_2-2) = 2t_2-2$. The rest of the claim follows by induction on $i$. For $C_M = C_K$, note that 
$$(2-t_2)(t_4-t_2-2) = 2t_2-2\in \sr{I}(C_N/C_K),$$
and therefore
$$(t_4-t_2-2) + 2t_2-2 = t_4+t_2-8\in \sr{I}(C_N/C_K).$$
As before, induction on the number of prime divisors of $\frac MK$ renders it sufficient to show that the base case holds for $M/K = q$ prime. Moreover, the above arguments imply that we need only show $t_4-t_2-2\in\burn(C_N/C_M)$. 

If $q = 2$, an application of \cref{eqn:norm} gives
$$\N{K}{M}(t_4-t_2-2) = 2t_8 - t_4 + 3t_2 - 2.$$
Thus
$$\N{K}{M}(t_4-t_2-2) - 2\tr{K}{M}(t_4-t_2-2) = t_4-t_2-2\in\burn(C_N/C_M),$$ as desired. When $q \neq 2$, we have 
$$ \N{K}{M}(2t_2-4) = 2t_2-4 + \frac{4^{q}-4}{q}t_q - \frac{4^q - 4}{2q}t_{2q},$$
and so
$$2t_2-4 = \N{K}{M}(2t_2-4) - \frac{4^{q-1}-1}{q}\tr{K}{M}(2t_2-4) \in \sr{I}(C_N/C_M)$$ and 
$$\N{K}{M}(t_4-t_2-2) = t_4 - t_2 -2 + \frac{-2^q+2}{q}t_q + \frac{-4^q + 2^q + 2}{2q}t_{2q} + \frac{4^q-4}{4q}t_{4q}.$$
Thus $t_4-t_2-2\in\sr{I}(C_N/C_M)$ if and only if
$$X := \frac{-2^q+2}{q}t_q + \frac{-4^q + 2^q + 2}{2q}t_{2q} + \frac{4^q-4}{4q}t_{4q} \in \sr{I}(C_N/C_M)$$
Furthermore, $X \in \sr{I}(C_N/C_M)$ if and only if
\begin{align*}
    Y &:= X - \frac{4^q-4}{4q}\tr{K}{M}(t_4-t_2-2)\\
      &= \frac{-4^q+2\cdot 2^q}{4q}t_{4q} + \frac{2\cdot4^q-4\cdot 2^q}{4q}t_{4q}\\
      &= \frac{-2^q}{4}\frac{2^{q-1}-1}{q} t_q ( 2t_2 - 4) \in \sr{I}(C_N/C_M).
\end{align*}
However, we know this element is in $\sr{I}(C_N/C_M)$ since we have already shown that the factor $2t_2 - 4 \in \sr{I}(C_N/C_M)$. Hence $t_4-t_2-2 \in \sr{I}(C_N/C_M)$ and the case $i = 0$ holds. 

Now let $1\leq i \leq j+2$ and suppose the claim holds for all $k < i-1$. We need to show that $t_{2^{i}} + t_2 - (2^{i}-2) \in \sr{I}(C_N/C_M)$. By the inductive hypothesis, we know $t_{2^{i-1}} + t_2 - (2^{i-1}-2) \in \sr{I}(C_N/C_{M/2})$ and $t_4-t_2-2\in \sr{I}(C_N/C_M)$. Therefore
\begin{align*}
    \tr{M/2}{M}(t_{2^{i-1}} + &t_2 - (2^{i-1}+2)) - (t_4-t_2-2) + (2^{i-2}+1)(2t_2 - 4) \\
    &= t_{2^{i}} + t_2 -2^i- 2 \in \sr{I}(C_N/C_M),
\end{align*}
completing the proof. 
\end{proof}


\section{The Trace Ideal}\label{sec:TI(Cn)}
The following section calculates the trace ideal $\TI{K}{F}$ for finite cyclic extensions and their profinite counterparts. In particular, we prove that finite extensions produce a principal trace ideal, and the profinite cases are colimits of these principal ideals. Our results coincide with those given by Epkenhans \cite[Proposition 5]{epkenhans:1998}, who calculates the classical absolute trace ideal $\mathcal{T}(G/G)$ for cyclic $2$-groups.

\subsection{Odd degree extensions} 
In the case of odd degree cyclic extensions, we can give a very nice presentation of the trace ideal. Taking advantage of the Tambara structure of the trace ideal allows us to greatly simplify our presentation. In particular, seen as a Tambara ideal of $\burn_{C_N}$, $\TI{K}{F}$ is strongly principal.

\begin{theorem}\label{thm:odd deg Cn extn}
Let $G=C_N$, where $N$ is odd with prime decomposition $p_1^{e_1}\cdots p_s^{e_s}$. Then the trace ideal has the level-wise description \begin{align*}
\TI{K}{F}(C_N/C_M) &= (t_i - i~:~ i\mid M)\\
&= (t_{p^k}-p^k ~:~ p \text{ prime, } p^k\mid M),
\end{align*} for all $C_M\leq C_N$. Moreover, $\TI{K}{F}$ is generated as a Tambara ideal by the element $$X := \sum_{i=1}^s (t_{p_i} - p_i)\in \burn(C_N/C_{\hat n})$$
where $\hat{n} = p_1\cdots p_s$
\end{theorem}
\begin{proof}
First note that any divisor $i\mid M$ must be odd. By \cref{thm:odd deg ker card},
\[
  \TI{K}{F}(C_N/C_M) = \left\{X=\sum_{i\mid M}a_it_{i} ~:~ \sum_{i\mid M}ia_i=0\right\}.
\]
Clearly then $t_i-i\in \TI{K}{F}(C_N/C_M)$ for any $i\mid M$. Moreover, for any $X=\sum_{i\mid M}a_it_{i}\in \TI{K}{F}(C_N/C_M)$, we have \[
X - \sum_{\substack{i\mid M\\ i\neq 1}}a_i(t_i-i)=a_1 +\sum_{\substack{i\mid M\\ i\neq 1}} ia_i =0,
\] which proves the first equality. The second equality follows from the fact that $t_jt_k = t_{jk}$ for $j, k$ relatively prime. To show $\TI{K}{F} = \gent{X}$ first note that for all $1\leq i\leq s$ we have $\res{p_i}{\hat{n}}(X) = t_{p_i}-p_i$. Applying \cref{lem:prime-gens} gives us the rest of the generators from the second equality. 
\end{proof}


\subsection{Cyclic 2-extensions} Having determined a single generator for the trace ideal for odd degree cyclic extensions, we will now consider cyclic 2-extensions. We find that the trace ideal for these extensions is still principal, but not always strongly principal. Nor is the ideal prime, as can be observed by comparing our result with the spectrum of cyclic $p$-groups calculated by Nakaoka \cite{nakaoka:2014}.

\begin{definition}\label{defn:tau}
Let $F$ be a field and $\alpha\in F^\times$. Define $\tau_F(\alpha)$ as
$$\tau_F(\alpha) = \left\{\begin{array}{cc}
         0 & \alpha \text{ not a sum of squares} \\
         1 & \alpha \in F^{\XX}\\
         2^n & \alpha \in D_F(2^n)\setminus D_F(2^{n-1}),\; n\geq 1
         \end{array}\right.$$
    where $D_F(m)$ is the set of sums of $m$ squares in $F$.
Hence $\tau_F(\alpha)$ is the least power of $2$ such that $\alpha$ is a sum of that many squares in $F$.
\end{definition}

\begin{remark}
The notation $D_F(m)$ is a derivative of the notation $D(q)$ for the set of all elements represented by a quadratic form $q$. Thus, we will use $D_F(m)$ and $D(m\langle 1\rangle_F)$ interchangeably.
\end{remark}

Note that $\tau_F(\alpha) = 0$ if and only if there is an ordering of $F$ such that $\alpha$ is negative (cf. \cite[p.378]{lam:2005}).

\begin{proposition}
If $\tau_F(\alpha) \neq 0$, then $\tau_F(\alpha)$ is the additive order of $\gen{1}-\gen{\alpha}$ in $GW(F)$. Moreover, $\tau_F(\alpha)$ is zero if and only if $\gen{1}-\gen{\alpha}$ has infinite order. 
\end{proposition}

\begin{proof}
It is well known that the torsion of $GW$ is 2-primary. Therefore, the additive order of $\gen{1}-\gen{\alpha}$ is a power of 2 or infinite. Clearly the order of $\gen{1}-\gen{\alpha}$ is $1$ precisely when $\tau_F(\alpha)=1$. Otherwise, suppose $o(\gen{1}-\gen{\alpha}) = 2^n$ for some $n\geq1$. That is, $2^n$ is the least power of two such that $2^n\gen{\alpha} = 2^n\gen{1}$, so $\alpha\in D(2^n\gen{1})$. Further, if $\alpha \in D(2^m\gen{1})$ for $m < n$, then $\alpha$ is a similarity factor of $2^m\gen{1}$ as Pfister forms are round forms (see \cite[Appendix to \S X.1]{lam:2005}). But then
$2^m\gen{\alpha} = \gen{\alpha}\otimes 2^m\gen{1} = 2^m\gen{1}$, contradicting the minimality of $n$. 
\end{proof}

\begin{theorem}\label{thm:c2}
If $G = C_2$, then
$\TI{K}{F}=\gent{\tau_F(\Delta)(t_{2,2}-2)}$ where $\Delta$ is the discriminant of $K/F$.
\end{theorem}

\begin{proof}
Since $\TI{K}{F}(C_2/e)=(0)$ (see \cref{rmk: TI G/e}), it only remains to investigate $\TI{K}{F}(C_2/C_2)$. For $X = mt_2 + n \in \TI{K}{F}(C_2/C_2)$, we must have $n = -2m$ by requirements on the dimension of $\Dr(X)$. It is well known that $\tr{F}{K}(\gen{1}) = \gen{2, 2\Delta}$ for quadratic extensions (see \cite[Lemma VII.6.17]{lam:2005}), so applying the Dress map yields $$\Dr(X) = m(\gen{2, 2\Delta} - \gen{1, 1}) = m\gen{2}(\gen{\Delta} -\gen{1}) = 0.$$ 
Since $\gen{2}$ is a unit, the order of $\Dr(X)$ is the same as that of $\gen{\Delta} -\gen{1}$ and hence $\tau_F(\Delta) \mid m$. Therefore every element of $\TI{K}{F}(C_2/C_2)$ is an integer multiple of $\tau_F(\Delta)(t_2-2)$. 
\end{proof}

Before examining the trace ideal for a general cyclic 2-extension, it is useful to calculate the case for a $C_4$ extension. Recall that a quadratic extension $F\subseteq E:=F(\sqrt{\Delta})$ embeds into a $C_4$ extension $K/F$ if and only if $\Delta=a^2+b^2$ for some $a,b\in F^{\times}$.
Moreover, we can write $K = F(\sqrt{\delta})$ where $\delta = x(\Delta - a\sqrt{\Delta})$ for some $x \in F^\times$ (cf. \cite[\S VII.6 and \S VIII.5]{lam:2005}).

\begin{theorem}\label{thm:c4}
Suppose $G = C_4$ and let $E = K^{C_2}$. Let $\Delta, a, b, x$ and $ \delta$ be given as above. If $K/F$ embeds into a cyclic extension of degree $8$, then $$\TI{K}{F} = \gent{t_{4, 4} - t_{4, 2} - 2}.$$ 
Otherwise, $$\TI{K}{F} = \gent{2t_{4, 2} - 4, \pi_F(x)(t_{4, 4}-4), \tau_E(\delta)(t_{2, 2}-2)},$$
where \[
\pi_F(x)=\left\{\begin{array}{cc}
     0 &  \tau_F(x)=0;\\
     2 & \tau_F(x)=1,2; \\
     \frac{\tau_F(x)}{2} & \tau_F(x)\geq 4.
\end{array}\right.
\]
\end{theorem}

\begin{proof}
Since $E/F$ is a quadratic Galois extension, \cref{thm:c2} tells us that \[
\TI{K}{F}(C_4/e) = (0) \hspace{.1in}\text{ and }\hspace{.1in} \TI{K}{F}(C_4/C_2) = (\tau_E(\delta)(t_2 - 2)).\]

Now let $X \in \TI{K}{F}(C_4/C_4)$. Since $\card(X) = 0$, $X$ is of the form $mt_4+nt_2-(4m+2n)$ for some $m, n \in \ZZ$. The Dress map takes this element to $$\Dr(X) = m\gen{1, \Delta, x, x} + n\gen{2, 2\Delta} - (4m+2n)\gen{1},$$ since $\Dr_{C_4}(1) = \gen{1}$, $\Dr_{C_4}(t_2) = \gen{2, 2\Delta}$, and $\Dr_{C_4}(t_4) = \gen{1, \Delta, x, x}$ (by \cite[Corollary VII.6.19]{lam:2005}).

When $m = 0$, our work on quadratic extensions implies $\Dr(X) = 0$ if and only if $\tau_F(\Delta)\mid n$. Hence any such $X$ is a multiple of $\tau_F(\Delta)(t_2-2)=2t_2-4$. 

If instead $n = 0$, then $\Dr(X) = m(\gen{\Delta, x, x} - \gen{1, 1, 1})$. If $\tau_F(x) \leq 2$, $\Dr(X)$ becomes $m(\gen{\Delta}-\gen{1})$ and so $\Dr(X) = 0$ if and only if $2\mid m$. Otherwise, if $\tau_F(x) \geq 4$, then $\Dr(X) = 2m(\gen{x}-\gen{1})$ and so $\tau_F(x)\mid 2m$ is equivalent to $\frac{\tau_F(X)}{2} \mid m$. In either case, $\pi_F(x)\mid m$ and hence all such elements are integer multiples of $\pi_F(x)(t_4-4)$.

Now suppose both $m,n$ are non-zero. Note that $m$ and $n$ must have the same parity since the determinant of $\Dr(X)$ must be $1$ for $X\in \TI{K}{F}$. Any element of the ideal with even $m$ and $n$ can be obtained from the two generators already given, so it just remains to investigate the case where both $m$ and $n$ are odd. We claim that for such an $X$, $$X \in \TI{K}{F}(C_4/C_4) \iff t_4-t_2-2 \in \TI{K}{F}(C_4/C_4).$$
We first show that $\pi_F(x) = 0$ implies that there can be no such $X\in \TI{K}{F}$. Suppose otherwise. Then by the ideal properties of the trace ideal, we get
\begin{align*}
    t_2X + 2m(2t_2-4) = 2mt_4 - 8m \in \TI{K}{F}.
\end{align*}
But then $0 \mid 2m$ by the previous paragraph, which is a clear contradiction. 
Now suppose $\pi_F(x) \neq 0$. Then since $\gcd(\pi_F(x), m) = 1$, there are some $a, b\in \ZZ$ such that $am + b\pi_F(x) = 1$. Thus $X \in \TI{K}{F}$ if and only if
\begin{align*}
    aX + b\pi_F(x)(t_4-4) - (n+1)(t_2-2) = t_4-t_2-2 \in \TI{K}{F},
\end{align*}
which proves the claim. 

Further, by \cite[Section 4, Proposition 9]{drees/epkenhans/kruskemper:1997}, $K/F$ embeds into a $C_8$ extension if and only if $\Dr(t_4-t_2-2)=\gen{1, \Delta, x, x} - \gen{2, 2\Delta, 1, 1}$ is zero. That is, if there is no such embedding, the trace ideal is generated by $\pi_F(x)(t_4-4),~ 2t_2-4\in \burn(C_4/C_4)$ and $\tau_E(\delta)(t_2 - 2)\in \burn(C_4/C_2)$ as claimed.

Suppose $K/F$ embeds into a $C_8$ extension. Then $\tau_E(\delta)=2$ and $t_4-t_2-2\in \TI{K}{F}(C_4/C_4)$ (so $\pi_F(x)\neq 0$). \cref{lem:2gens} implies that this element generates the entire ideal. 
\end{proof}

\begin{lemma}\label{lem:2-group upper bound}
Suppose $G = C_{2^n}$ for $n \geq 2$. Then for all $1\leq m\leq n$,
$$\TI{K}{F}(C_{2^n}/C_{2^m}) \subseteq (t_{2^i} + t_2-2^i-2 \mid 1 \leq i \leq m),$$
and hence
$$\TI{K}{F} \subseteq \gent{t_{4, 4} -t_{4, 2}-2}.$$
\end{lemma}

\begin{proof}
Suppose $X = \sum_{i = 0}^{n-m}a_it_{2^i}\in \TI{K}{F}(C_{2^n}/C_{2^m})$ for $a_i \in \ZZ$. Then $\sum_{i = 0}^{n-m}2^ia_i = 0$ since $\card(X) = 0$. Now, from basic Galois and quadratic form theory we know that the determinant $d(\Dr(t_i))$ has the same square class as the discriminant of $K^{C_{2^m}} \subseteq K^{C_{2^{m-i}}}$ and therefore $d(\Dr(t_i)) = d(\Dr(t_j))$ for all $1\leq i, j \leq n-m$. Thus the condition that $d(\Dr(X)) = 1F^\boxtimes$ implies that $\sum_{i = 1}^{n-m}a_i \equiv_2 0$. One can easily check that any $X$ satisfying these two conditions can be written as a sum of the generators listed in the theorem. The second part of the theorem follows directly from \cref{lem:2gens}. 
\end{proof}

Now let $G = C_{2^n}$ for some $n\geq 3$, $E = K^{C_2}$, and $L = K^{C_4}$. Since $K/L$ is a $C_4$ extension, we can write $E = L(\sqrt{\Delta})$ for some $\Delta = a^2+b^2\in L$ and $K = E(\sqrt{\delta})$ where $\delta = x(\Delta+a\sqrt{\Delta})$. 

\begin{theorem}\label{thm:twopowers}
Adopting the notation above, we have the following presentation of the trace ideal:
\begin{enumerate}
    \item If K/L embeds into a $C_8$ extension,
    $$\TI{K}{F} = \gent{t_{4, 4} - t_{4, 2} - 2}.$$
    \item If $\tau_E(\delta) = 0$,
    $$\TI{K}{F} = \gent{t_{8, 4} - t_{8, 2} - 2}.$$
    \item Suppose neither of the above hold. Let $m$ be the minimal index (with $3\leq m\leq n$) such that $\tr{L}{K^{C_{2^m}}}(\gen{\Delta, x, x} - \gen{2, 2\Delta, 1}) = 0$. Then $\TI{K}{F}$ is the Tambara ideal generated by\begin{enumerate}
        \item $t_{2^{m}, 2^{m}} - t_{2^{m}, 2^{m-1}} - 2t_{2^{m}, 2^{m-2}},$
        \item $t_{8, 4} - t_{8, 2} - 2,$
        \item $a_i(t_{{2^{i}}, 2^{i}} - 2^{i})$
    \end{enumerate} for some $a_i \in \NN$ where $i$ ranges from 0 to $n-1$. If there is no such $m$, then take $(a)$ to be $0$. Furthermore, each $a_i = o(\tr{K}{K^{C_i}}(\gen{1}) - 2^i\gen{1})$ is a power of two ($a_i\neq 1$) with $a_{i+1} \mid a_i \mid 2a_{i+1}$.
\end{enumerate}
\end{theorem}

\begin{remark}
It is worth noting that, for (3), we are unsure as to whether such a minimal index $m$ ever exists. That is, it could be that (a) should always be taken to be $0$.
\end{remark}

\begin{proof}
(1) By considering the sub-extension $K/L$, \cref{thm:c4} tells us that $t_{k, 4} - t_{k, 2} - 2 \in \TI{K}{F}$. Combining this observation with \cref{lem:2-group upper bound} yields the desired result. 

\vspace{.2cm}

\noindent (2) Consider the sub-extension $K^{C_8}/E$. Then $t_{4} - t_{2} - 2 \in \TI{K}{F}(C_{2^n}/C_8)$, and since the restriction of this element to $C_4$ is $2t_2-4$, \cref{thm:c4} tells us that $\TI{K}{F}(C_{2^n}/C_{2^i})$ is as desired for $i = 0,1, 2$. 

Now let $i \geq 3$ and consider $$X = \sum_{j = 0}^im_jt_{2^j} \in \TI{K}{F}(C_{2^n}/C_{2^i}).$$ Then $\res{2}{{2^i}}(X) = m_i(t_2 - 2)$, but the fact that $\TI{K}{F}(C_{2^n}/C_{2^i}) = (0)$ implies $m_i = 0$. So we have $X = \sum_{j = 0}^{i-1}m_jt_{2^j}$. The conditions on the determinant and dimension of $\Dr(X)$ imply that $\sum_{j = 0}^{i-1}m_j = 0$ and $\sum_{j = 1}^{i-1}m_j \equiv_2 0$. Thus $$\TI{K}{F}(C_{2^n}/C_{2^i}) \subseteq (t_2^j + t_2 - (2^j+1))$$ where $j$ ranges from $1$ to $i-1$. Since $t_{l, 4} - t_{l, 2} - 2 \in \TI{K}{F}$, \cref{lem:2gens} implies that each of these generators is in $\TI{K}{F}$. Therefore $\TI{K}{F} = \gent{t_{l, 4} - t_{l, 2} - 2}$ as desired.

\vspace{.2cm}

\noindent (3) Supposing that $\tau_E(\delta) \neq 0$, let $a_1 = \tau_E(\delta)$ and $a_2 = \pi_L(x)$. First suppose there is no $i$ for which $\tr{L}{K^{C_m}}(\gen{D,x,x}-\gen{2,2D,1})=0$. Then $\TI{K}{F}(C_{2^n}/C_{2^i})$ is clearly as claimed for $i = 0, 1, 2$. 

Now let $3\leq i\leq n$ and suppose $\TI{K}{F}(C_{2^{n}}/C_{2^{i-1}})$ is as claimed. By \cref{lem:2gens}, we have $t_{2^j} + t_2 - (2^j+2) \in \TI{K}{F}$ for all $1\leq j < i$.  Consider $X = \sum_{j=0}^im_jt_{2^j} \in \TI{K}{F}(C_{2^n}/C_{2^i})$. Supposing $m_{i}$ is even, the condition on the determinant of $\Dr(X)$ implies that $\sum_{j=0}^{i-1}m_j \equiv_2 0$. Therefore $X$ is in the trace ideal if and only if $m_i(t_{2^i} - 2^i) \in \TI{K}{F}$. Taking $$a_i = o(\tr{K}{K^{C_{2^i}}}\gen{1} - 2^i\gen{1})$$ shows that $X \in \TI{K}{F}$ if and only if $a_i \mid m_i$, which is to say that $X$ is in the ideal generated as claimed. Observing that $\res{{2^{i-1}}}{{2^i}}(a_i(t_{2^i} - 2^i)) = 2a_i(t_{2^{i-1}} - 2^i) \in \TI{K}{F}$ implies $a_{i-1} \mid 2a_i$. Furthermore, $\tr{{2^{i-1}}}{{2^i}}(a_{i-1}(t_{2^{i-1}} - 2^{i-1})) + 2^{i-1}(t_2-2) = a_{i-1}(t_{2^i} - 2^i) \in \TI{K}{F}$ implies that $a_i \mid a_{i-1}$. Since $t_F(\delta) = a_1 \neq 0$ by assumption, these relations imply that $a_i \neq 0$, and the condition on the determinant implies that $a_i \neq 1$. 

Now suppose that $m_i$ is odd. Since $\gcd(m_i, a_i) = 1$ there is some integer combination of $X$ and $a_i(t_{2^i} - 2^i)$ for which the coefficient of $t_{2^i}$ is $1$. Since we have shown $a_i(t_{2^i} - 2^i)\in \TI{K}{F}$ and moreover $t_{2^j} + t_2 - 2^j - 2 \in \TI{K}{F}$ for all $1 \leq j < i$, we see that $X \in \TI{K}{F}$ if and only if $Y := t_{2^i} - t_{2^{i-2}} - 2t_{2^{i-3}}$ is in the trace ideal as well. But $Y = \tr{C_4}{C_i}(t_4-t_2-2)$, and so $\Dr(Y) = \tr{L}{K^{C_i}}(\gen{\Delta, x, x} - \gen{2, 2\Delta, 1})$. Since we have supposed there is no $i$ for which $\Dr(Y)=0$, we have no such $X$ in the trace ideal. 

Otherwise, let $m$ be the minimal such index. Then the arguments above apply for all indices $i < m$, so we assume $i = m$. By the arguments above, it suffices to take $Y$ as an additional generator and hence $\TI{K}{F}(C_{2^n}/C_{2^m})$ is as claimed. Furthermore, an obvious adaptation of the argument from \cref{lem:2gens} shows that the level-wise ideals are as claimed for all $i > m$. \end{proof}


\subsection{General cyclic extensions} Combining the results from the previous two subsections allows us to consider an arbitrary cyclic group $C_N$. \cref{thm:odd deg Cn extn} determines the trace ideal when $N$ is odd, so we examine $N$ even.

\begin{theorem}\label{thm:TI for Cn}
Let $G = C_N$ where $N$ has prime decomposition $2^\mu p_1^{\sigma_1}\cdots p_s^{\sigma_s}$ for $\mu, \sigma_i\geq 1$. Then $\TI{K}{F}$ is generated by 
\begin{enumerate}
    \item $\sum_{i=1}^s (t_{p_i}-p_i) \in \burn(C_N/C_{p_1\cdots p_s})$ and
    \item $\sr{G}$, where $\sr{G}$ generates $\TI{K}{K^{C_{2^\mu}}}$ as in \cref{thm:twopowers}.
\end{enumerate}
\end{theorem}

\begin{proof}
Let $\mathcal{I}$ denote the ideal generated by the elements from the theorem statement and let $X = \sum_{i=1}^s (t_p-p) \in \mathcal{I}(C_N/C_{p_1\cdots p_s})$. Note that $\res{p_i}{p_1\dots p_s}(X) = t_{p_i}-p_i$ for all $1 < i < s$, so by \cref{lem:prime-gens}, $t_{p_i^j} - p_i^j$ is in each level of $\mathcal{I}$ where this element makes sense. Furthermore, this shows that $t_m-m$ for $m$ odd is in each level of $\mathcal{I}$ where it makes sense. Multiplying this element by $t_{2^j}$, we obtain $t_{2^jm}-mt_{2^j}$. 

Let $X = \sum_{i \mid M}m_it_i \in \TI{K}{F}(C_N/C_M)$ for some $M \mid N$. We wish to show that $X\in\mathcal{I}$ as well. Let $2^k$ be the largest power of two dividing $M$ and let $m' = \frac{M}{2^k}$. Since all divisors of $m'$ will be odd, we can consider the element \[Y:=X-\sum_{j=0}^k\sum_{i\mid m'}m_{2^ji}(t_{2^ji}-it_{2^j}) \in \TI{K}{F}(C_N/C_M).\]

Since each summand is in $\mathcal{I}(C_N/C_M)$, it is sufficient to show that we have $Y\in \mathcal{I}(C_N/C_M)$. Write $Y = \sum_{i=0}^kn_it_{2^i}$ for some $n_i \in \ZZ$. Thus we need to show that $Y \in \mathcal{I}(C_N/C_m)$. We know that $\res{2^k}{M}(Y) = \sum_{i=0}^kn_it_{2^i} \in \TI{K}{F}(C_N/C_{2^k})$, and since $\TI{K}{F}(C_N/C_{2^k}) = \mathcal{I}(C_N/C_{2^k})$ by definition, this element is some combination of the generators in $\sr{G}$. Hence we need only show that the element $\sum_{i = 0}^ja_jt_{N/2^jm', 2^i}$ is in $\mathcal{I}$ for each generator $ \sum_{i = 0}^ja_jt_{N/2^j, 2^i}$. 

Note that \cref{lem:prime-gens} and \cref{lem:2gens} imply the desired result when $\sr{G}$ is given by case (1) or (2) of \cref{thm:twopowers}. Therefore we need only consider (3). However, \cref{lem:prime-gens} and \cref{lem:2gens} imply the desired result for all generators other than $t_{2^{m}, 2^{m}} - t_{2^{m}, 2^{m-1}} - 2t_{2^{m}, 2^{m-2}}$ where $m$ is as in \cref{thm:twopowers}. But a straightforward adaptation of the argument from \cref{lem:2gens} applies to this case as well, which completes the proof. 
\end{proof}

\begin{corollary}
For $N$ given as above, the absolute trace ideal can be calculated as 
$$\mathcal{T}_{C_N} = \gent{X}$$
where $X = \sum_{p \text{ odd}}(t_{\hat{n}, p}-p) + Y$, $\hat{n} = 2^\lambda p_1\cdots p_s$, $\lambda = \min\{3, \mu\}$ and 
$$ Y = \left\{\begin{array}{cc}
         0 & \lambda = 0; \\
         2t_{2,2}-4 & \lambda = 1;\\
         t_4-t_2-2 & {\rm otherwise.}
         \end{array}\right.$$
\end{corollary}

\begin{proof}
A theorem of Epkenhans says that the absolute trace ideal is an intersection of finitely many trace ideals \cite{epkenhans:1999}. This, along with our computations of the trace ideal and Epkenhans computation of $\mathcal{T}(G/G)$ imply the desired result. 
\end{proof}

\section{The Trace Ideal for Profinite Extensions}\label{sec:profinite}
Let $\ZZ_p$ and $\hat{\ZZ}$ denote the $p$-adic integers and the profinite completion of the integers, respectively. By our work in the previous section, we can compute the trace ideal for $K/F$ where $\Gal(K/F)$ is either of these profinite groups.

\subsection{Tambara maps in the Burnside functor} For $K \mid  M \in \ZZ$, let $k = \frac M K$. We denote the $M\hat{\ZZ}$-set $M\hat{\ZZ}/K\hat{\ZZ}$ by $t_{M, k}$, and again the first index will be dropped when it can be inferred from context. With this identification, we clearly obtain the same formula for multiplication
 \begin{equation}
    t_{k}t_{j} = \text{gcd}(k, j)t_{\text{lcm}(k, j)}.
 \end{equation}

The Tambara maps given in Equations (3.3--3.5) follow analogously. Specifically, for arbitrary elements $X=\sum_{i=1}^{n_1} a_it_{i}\in \burn(\hat\ZZ/N\hat \ZZ)$ and $Y=\sum_{i=1}^{n_2} b_it_{i}\in \burn(\hat\ZZ/M\hat\ZZ)$ for some $n_1,n_2\in\NN$, we have\begin{align}
    \res{N}{M}(Y)&:= \res{N\hat\ZZ}{M\hat\ZZ}(Y)=\sum_{i=1}^{n_2} b_id_it_{ \frac{i}{d_i}},\\
    \tr{N}{M}(X)&:= \tr{N\hat\ZZ}{M\hat\ZZ}(X) = \sum_{i=1}^{n_1} a_it_{ik},\\
    \N{N}{M}(X)&:=\N{N\hat\ZZ}{M\hat\ZZ}(X) = \sum_{i=1}^{n_2}\frac{C(i)}{i}t_{i},
\end{align} where $d_i={\rm gcd}(i,k)$ and \[C(i) = \Bigg(\sum_{j\mid \frac{{\rm lcm}(i, k)}{k}}ja_j\Bigg)^{{\rm gcd}(i, k)} - \sum_{j \mid i,\; j < i}C(j)\] as before. Conjugation is again trivial.

Identifying the $p^n\ZZ_p$-set $p^n\ZZ_p/p^m\ZZ_p$ with $t_{p^n, p^{m-n}}$ allows us to similarly adapt these formulas for the $p$-adic integers. Multiplication carries over exactly, while the restriction, transfer and norm formulas translate by summing over powers of $p$. 

\subsection{The trace ideal for profinite groups} Before examining $\hat\ZZ$, it is enlightening to examine the $p$-adic case. By the identification given above, we get an analogous description of the trace ideal for $p$ odd as in \cref{thm:odd deg Cn extn}. When $p=2$, our description is comparable to that of case (1) from \cref{thm:twopowers}.

\begin{theorem}\label{thm:p-adic trace ideal}
Let $G=\ZZ_p$ for some prime $p$. If $p$ is odd, the trace ideal has level-wise description given by\[
\TI{K}{F}(\ZZ_p/p^n\ZZ_p) = ( t_{p^m} - p^m \;:\; m\in \NN),
\] which implies
\[
\TI{K}{F} = \gent{t_{p^i, p}-p~:~i\in\NN}.
\]

For $p = 2$, the trace ideal has the level-wise description
\[
\TI{K}{F}(\ZZ_2/2^n\ZZ_2) = ( t_{2^m} +t_2- 2^m-2 \;:\; m\in \NN),
\] which implies
\[
\TI{K}{F} = \gent{ t_{2^i, 4} - t_{2^i, 2} - 2 ~:~i\in \NN}.
\]
\end{theorem}
\begin{proof}
When $p$ is an odd prime, the proof of the level-wise description is analogous to the one given in \cref{thm:odd deg Cn extn}. The second description follows from \cref{lem:prime-gens}. When $p=2$, let $\mathcal{I}$ be the ideal claimed in the theorem statement. By the same arguments as in \cref{lem:2-group upper bound} we have that $\TI{K}{F}\subseteq \mathcal{I}$. To show the other inclusion note that for all $j$, $K^{2^j\ZZ_2}/K^{2^{j+2}\ZZ_2}$ embeds into a $C_8$ extension. Thus by \cref{thm:c4} we have that $t_4-t_2-2\in\TI{K}{F}(\ZZ_2/2^j\ZZ_2)$. An application of \cref{lem:2gens} shows that the rest of the generators are in $\TI{K}{F}$. 
\end{proof}

\begin{theorem}\label{thm:alg closure}
Let $K/F$ be a Galois extension with $G=\hat\ZZ$. Then $\TI{K}{F}$ is generated by $t_{2^n,4}-t_{2^n,2}-2\in\burn(\hat\ZZ/2^n\hat\ZZ)$ and $t_{p^n, p}-p\in \burn(\hat\ZZ/p^n\hat\ZZ)$ for odd primes $p$ and all $n\in \NN$. 
\end{theorem}
\begin{proof}
Let $\mathcal{I}$ be the Tambara ideal generated by the elements from the theorem statement. For a prime $p$, the $p$-adic extensions are sub-extensions of the algebraic closure, so \cref{thm:p-adic trace ideal} tells us we have $t_{p^m,p^n}-p^n\in\mathcal{I}(\hat{\ZZ}/p^m\hat{\ZZ})$ for all odd $p$ and $n\geq 1, m \in \NN$. Note that this includes $\mathcal{I}(\hat{\ZZ}/\hat{\ZZ})$, the case where $m=0$. To show that $t_{p^n}-p^n \in \mathcal{I}(\hat{\ZZ}/m\hat{\ZZ})$, let $k$ be the largest power of $p$ dividing $m$. Then $\frac{m}{p^k}$ is relatively prime to $p$, so the restriction is $\res{m\ZZ}{p^k\ZZ}(t_{m,p^n}-t_{m,p}) = t_{p^k,p^n}-t_{p^k,p}$. The same argument works for the generators associated with $p=2$.
\end{proof}

The trace ideal for these profinite groups is clearly not finitely generated, as $t_4-t_2-2$ must be in each level of the trace ideal yet is not in the image of the restriction map. However, if we define $\sr{I}_n$ to be the ideal generated by $t_{n, 4}-t_{n, 2}-2$ for $G = \ZZ_2$ and the ideal generated by $t_{n, p} - p$ for $G = \ZZ_p$, we see that $\sr{I}_1\subseteq \sr{I}_2\cdots \subseteq \sr{I}_n \subseteq \cdots$ and $\TI{K}{F} = \bigcup\sr{I}_n$. For the case of $G = \hat{\ZZ}$, let $n \geq 4$, $m = n!$, and take $\sr{I}_n$ be the ideal generated by the element $t_{m, 4}-t_{m,2}-2 +\sum_{p < n}t_{m, p}-p$. Then we similarly have $\sr{I}_4\subseteq \sr{I}_5\cdots \subseteq \sr{I}_n \subseteq \cdots$ and $\TI{K}{F} = \bigcup\sr{I}_n$. So in all cases considered, $\TI{K}{F}$ is the union of an ascending chain of strongly principal ideals.


\section{Some Applications and Examples}\label{sec:applications}

We can apply our computations to some examples of interest. In particular, \cref{thm:c2} allows us to determine the trace ideal for common quadratic extensions. We can also use the trace ideal of a quadratic extension to gain insight into the structure of the base field. Finally, we can apply our results to completely describe the trace ideal for both finite and profinite extensions of finite fields.

\subsection{Quadratic extensions} Recall that our characterization of the trace ideal for a cyclic 2-extension $F(\sqrt{\alpha})/F$ depends on the number such that $\alpha$ is a sum of that many squares. We thus get the following direct corollaries of \cref{thm:c2}:
\begin{corollary}
We have the following computations:
\begin{enumerate}
    \item $\TI{\CC}{\RR} = \gent{0}$ implying by \cref{rem:surjective} that $\burn_{C_2}\cong \tamb{GW}_{\RR}^{\CC}$.
    \item For $r=\frac{a}{b}\in\QQ$ a non-square with $ a,b\in\ZZ$, 
    $$\TI{\QQ(\sqrt{r})}{\QQ} = \left\{\begin{array}{cc}
         \gent{0} & r < 0; \\
         \gent{2t_{2,2}-4} & ab = x^2+y^2,\;\; x, y \in \ZZ; \\
         \gent{4t_{2,2}-8} & {\rm otherwise.}
         \end{array}\right.$$
    \item For a finite field $\FF_q$, 
    $$ \TI{\FF_{q^2}}{\FF_{q}} = \gent{2t_{2,2}-4}.$$
\end{enumerate}
\end{corollary}
\begin{proof}
Part (1) is immediate by noting that $-1$ is negative and therefore not a sum of squares. Part (2) follows from the fact that any positive rational $\frac ab$ is the sum of four squares, and $\frac ab$ is a sum of two squares if and only if $ab$ is as well. Finally, (3) follows from recalling that every element of a finite field is the sum of two squares.
\end{proof}

This result also yields a characterization of Pythagorean and formally real fields in terms of the trace ideals they admit for quadratic extensions. 

\begin{corollary} Let $F$ be a field. Then
\begin{enumerate}
    \item $F$ is formally real if and only if $\TI{K}{F} = \gent{0}$ for some quadratic extension $K/F$.
    \item $F$ is Pythagorean if and only if $\TI{K}{F} = \gent{0}$ for all quadratic extensions $K/F$.
\end{enumerate}
\end{corollary}
\begin{proof} 
(1) The trace ideal is zero for an extension $F(\sqrt{\alpha})/F$ when $\alpha$ is not a sum of squares in $F$. If there is such an $\alpha$, then $F$ is formally real. If instead we suppose that $F$ is formally real, we can take $\alpha=-1$.

(2) Recall that in a Pythagorean field, any sum of squares is itself a square. Hence if we have a quadratic extension of a Pythagorean field, the discriminant cannot be a sum of squares. Similarly, if $\TI{F(\sqrt{\alpha})}{F} = \gent{0}$ for every quadratic extension $F(\sqrt{\alpha})/F$, then every non-square $\alpha\in F^\times$ is not a sum of squares.
\end{proof}

\subsection{Finite fields} The work of the previous sections allows us to give a complete description of the trace ideal for extensions of finite fields. We let $\FF_q$ denote the finite field with $q$ elements, where $q$ is a power of an odd prime. The Grothendieck-Witt ring on $\FF_q$ has a particularly unique structure, which makes this family of fields a rich source of study. 

Over a finite field, every quadratic form is universal \cite[Proposition II.3.4]{lam:2005} and so is completely specified in $GW(\FF_q)$ by its dimension and determinant \cite[Theorem II.3.5(1)]{lam:2005}. Recall that the determinant is determined up to square class $F^{\times}/F^{\XX}$, and moreover $\abs{\FF_q^{\times}/\FF_q^{\XX}}=2$ (cf. \cite[\S II.3]{lam:2005}). We denote the two square classes by $1$ and $\alpha$. Note that $\alpha$ is a sum of two squares, and we may take $\alpha=-1$ if and only if $q\equiv 3\pmod 4$. In any case, we have $GW(\FF_q)\cong \ZZ \oplus \FF_q^\times/\FF_q^{\XX}$ where the isomorphism is given by $\gen{a_1,\dots,a_n}\mapsto (n, ~ \prod_i a_i\FF_q^{\XX})$.

This simplified presentation of the Grothendieck-Witt ring permits an explication of the Tambara structure of $\tamb{GW}$ on a finite field, the details of which are worked out in 
\cref{app:HM} by H.~Chen and X.~Chen. These computations help us describe the Dress map and the trace ideal for finite fields. In particular, any finite extension $\FF_{q^N}/\FF_q$ will have Galois group $C_N$. The Dress map sends 
\[
X=\sum_{i\mid M}a_it_i \in \burn(C_N/C_M) \longmapsto \bigg(\sum_{i}ia_i,~\prod_{i~{\rm even}}\alpha^{a_i}\bigg) \in \tamb{GW}(\FF_{q^M})
\] where $\alpha$ generates the non-square class of $\FF_{q^M}$, and applying \cref{thm:TI for Cn} then gives us a complete description of the trace ideal. In particular, we find that the trace ideal is strongly principal in this case.

Now consider $\FF_q$ inside of its quadratic or algebraic closure, which have Galois groups $\ZZ_2$ and $\hat{\ZZ}$, respectively. The Dress map is described by a similar formula as given above, and we can apply the work of \cref{sec:profinite} to describe the trace ideal. In both these profinite cases, the top field is clearly quadratically closed, so \cref{rem:surjective} applies.

\begin{theorem}
\label{thm:finite fields}
Let $Q$ and $K$ denote the quadratic and algebraic closure of $\FF_q$, respectively. Then we have the following calculations: \begin{enumerate}
    \item   Suppose $N$ has prime decomposition $2^{\mu}p_1^{\sigma_1}\cdots p_m^{\sigma_m}$ and take $\hat N=2^{\hat\mu}p_1\cdots p_s$ with \[\hat\mu=\left\{\begin{array}{cc}
     \mu&  \mu=0,1;\\
     2& \mu\geq 2.
\end{array}\right.\] Then \[
\TI{\FF_{q^N}}{\FF_q} = \bigg(\!\!\!\bigg(X_2 + \sum_{i=1}^s (t_{\hat N, p^i} -  p_i)\bigg)\!\!\!\bigg), 
\] where\[
X_2 = \left\{\begin{array}{cc}
         2\mu(t_{\hat N,2} - 2) & \mu = 0,1; \\
         t_{\hat N, 4}-t_{\hat N, 2} - 2
         & \mu \geq 2.
         \end{array}\right.
\]
    \item $\TI{Q}{\FF_q} = \gent{ t_{2^i, 4} - t_{2^i, 2} - 2 }$ where $i$ ranges over all of $\NN$.
    \item $\TI{K}{\FF_q}$ is the ideal given in \cref{thm:alg closure}.
\end{enumerate}
\end{theorem}
\begin{proof}
Part (1) follows from \cref{thm:TI for Cn} and taking appropriate restrictions of the stated generator (as in the proof of \cref{thm:odd deg Cn extn}). Note that if $\mu\geq 2$, we are in case (1) of \cref{thm:twopowers} since we can always embed $\FF_{q^N}/\FF_q$ appropriately. Parts (2) and (3) follow directly from \cref{thm:p-adic trace ideal} and \cref{thm:alg closure}, respectively.
\end{proof}

\appendix 
\crefalias{section}{appsec}
\section{Norms of quadratic forms over finite fields}\label{app:HM}

\vspace{-0.25cm}

\section*{\small{by Harry Chen and Xinling Chen}}

\newcommand{\Res}{\operatorname{Res}}
\renewcommand{\res}{\operatorname{res}}
\renewcommand{\tr}{\operatorname{tr}}
\renewcommand{\N}{\operatorname{N}}

\maketitle
Restriction, Scharlau transfer (with respect to field trace), and the Rost norm give the Grothendieck-Witt ring the structure of a Tambara functor \cite{bachmann:2016}.  Working over a finite base field, the values of $\underline{GW}$ and its restriction and transfer maps are known classically.  Meanwhile the Rost norm has only been computed explicitly relative to quadratic extensions \cite{wittkop:2006}.  Leveraging the Dress map and the structure of the Burnside Tambara functor for cyclic groups, we completely determine the Rost norm for any extension of finite fields with odd characteristic in Theorem \ref{thm:main} below. Since the absolute Galois group of $\mathbb{F}_q$ is $\widehat{\mathbb{Z}}$, one may view this result as complementary to \cref{thm:alg closure}.

Let $F=\FF_q$ be the finite field with $q=p^k$ odd. Recall that dimension and determinant form a ring isomorphism
\[
  GW(F)\cong \ZZ\oplus F^\times/F^\boxtimes,
\]
where the right-hand side has trivial multiplication on $F^\times/F^\boxtimes$. As such, every $n$-dimensional form in $GW(F)$ can be written as  either $n\<1\>$ or $(n-1)\<1\>\oplus\<\alpha\>$, where $\alpha$ is a generator of $F^\times$.  Following, \emph{e.g.}, \cite{lam:2005}, it is easy to write down the effect of restriction and transfer on these classes.

\begin{theorem}[Restriction and transfer for finite fields]\label{thm:res and tr for finite field}
Let $F=\FF_q\subseteq \FF_{q^m}=E$ and fix generators $\alpha\in F^\times$, $\beta\in E^\times$. Then
\begin{align*}
\res_F^E\<1\>&=\<1\>\text{,}\\
\res_F^E\<\alpha\>&=\begin{cases}\<1\>&\text{if $m$ is even,}\\\<\beta\>&\text{if $m$ is odd,}\end{cases}\\
\tr_F^E\<1\>&=\begin{cases}(m-1)\<1\>\oplus\<\alpha\>&\text{if $m$ even,}\\
m\<1\>&\text{if $m$ is odd,}\end{cases}\\
\tr_F^E\<\beta\>&=\begin{cases}m\<1\>&\text{if $m$ is even,}\\
(m-1)\<1\>\oplus\<\alpha\>&\text{if $m$ is odd.}\end{cases}
\end{align*}
\end{theorem}
The Rost norm is a multiplicative map $\N_F^E:GW(E)\to GW(F)$ that takes any unary form $\langle a\rangle$ to $\N_F^E(\langle a\rangle) = \langle N_{E/F}(a)\rangle$ where $N_{E/F}:E^\times\to F^\times$ is the classical field norm.  In order to determine the value of $\N_F^E$ on higher-dimensional forms, we need to know how it interacts with summation.  The following theorem of M.~Hill and K.~Mazur gives a general formula for this interaction when the group of equivariance is finite Abelian.

\begin{theorem}[Tambara reciprocity for finite Abelian groups {\cite[Theorem 2.5]{mazur:2019}}]\label{thm:HM} Let $G$ be a finite Abelian group and let $\underline S$ be a $G$-Tambara functor. for all $H<G$ and $a,b\in\underline S(G/H)$ 
\[
\N_H^G(a+b)=\N_H^G(a)+\N_H^G(b)+\sum_{H< K< G}\tr_K^G\bigg(\sum_{k=1}^{i_K}\N_H^K((\underline{ab})_k^K)\bigg)+\tr_H^G(g_H(a,b))
\]
where $i_K$ is the number of orbits of functions from $G/H$ to $\{a,b\}$ with stabilizer exactly $K$, and $(\underline{ab})_k^K$ is a monomial in some of the $W_G(K)$-conjugates of $a$ and $b$, and $g_H(a,b)$ is a polynomial in some of the $W_G(H)$-conjugates of $a$ and $b$.
\end{theorem}

This leads to a far more explicit formula for $\underline S = \underline{GW}$ for an odd prime extension of finite fields.

\begin{lemma}\label{lemma:TR}
Let $F=\FF_q\subseteq \FF_{q^\ell}=E$ where $\ell>2$ is prime. Then for all $a,b\in GW(E)$,
\[
\N_F^E(a+b)=\N_F^E(a)+\N_F^E(b)+\tr_F^E\bigg(\sum_{i=1}^{\ell-1}\frac{{\binom{\ell}{ i}}}{\ell}a^ib^{\ell-i}\bigg).
\]
\end{lemma}

\begin{proof}
Let $G = \operatorname{Gal}(E/F) = \langle \varphi\rangle\cong C_\ell$ where $\varphi$ is the Frobenius homomorphism.  It suffices to determine $g_e(a,b)$ from Theorem \ref{thm:HM}. According to \cite[Corollary 2.6]{mazur:2019}, we have
\[
  g_e(a,b) = \sum_{f\in \mathcal{I}/G}\prod_{i=0}^{\ell-1}\varphi^i(f(\varphi^i))
\]
where $\mathcal{I}$ is the set of \emph{nonconstant} functions $f\colon G\to \{a,b\}$ with the natural action of $G$.  Since $\varphi$ acts trivially on $GW(E)$, we see that we are just adding up all degree $\ell$ ordered monomials in $a$ and $b$ (which are not $a^\ell$ or $b^\ell$) up to cyclic permutation of factors.  Combining like terms, we have
\[
  g_e(a,b)=\sum_{i=1}^{\ell-1}\frac{{\binom{\ell} {i}}}{\ell}a^ib^{\ell-i}
\]
\end{proof}

\begin{theorem}[Norms for prime extensions of finite fields]\label{thm:prime}
Let $F=\FF_q\subseteq \FF_{q^\ell}=E$ for $\ell$ any prime and fix generators $\alpha\in F^\times$, $\beta\in E^\times$. Then
\begin{align*}
\N_F^E(n\<1\>)&=
  \begin{cases}
    n^\ell\<1\> &\text{if $\ell>2$},\\
    (n^2-1)\<1\>\oplus\<\alpha^{\frac{n^2-n}{2}}\> &\text{if $\ell=2$},
  \end{cases}\\
\N_F^E((n-1)\<1\>\oplus\<\beta\>)&=
  \begin{cases}
    (n^\ell-1)\<1\>\oplus\<\alpha^n\>, &\text{if $\ell>2$},\\
    (n^2-1)\<1\>\oplus\<\alpha^{\frac{n^2-3n}{2}}\>,&\text{if $\ell=2$}.
  \end{cases}
\end{align*}
\end{theorem}

\begin{proof}
First consider the effect of $\N_F^E$ on $n\langle 1\rangle$.  The Dress map gives us the commutative diagram
\begin{center}\begin{tikzcd}
\underline A(C_\ell/C_\ell)\arrow[r, "D"] & GW(\FF_{q})\\
\underline A(C_\ell/e)\arrow[r, swap, "D"] \arrow[u, "\N"]& GW(\FF_{q^\ell})\arrow[u, swap, "\N"]
\end{tikzcd}\end{center}
which, via the formulas in \cite{nakaoka:2014}, has the following effect on $n\in \underline A(C_\ell/e)$:
\begin{center}\begin{tikzcd}
n+\frac{n^\ell-n}{\ell}C_\ell/e\arrow[r, mapsto, "D"] & n\<1\>\oplus\frac{n^\ell-n}{\ell}\tr_F^E\<1\>\\
n\arrow[r, mapsto, swap, "D"] \arrow[u, mapsto, "\N"]& n\<1\>.\arrow[u, mapsto, swap, "\N"]
\end{tikzcd}\end{center}
As such, we know that
\[
\N_F^E(n\<1\>)=n\<1\>\oplus\frac{n^\ell-n}{\ell}\tr_F^E\<1\>
\]

We separately discuss the cases $\ell>2$ and $\ell=2$. When $\ell>2$, we know
\[
\N_F^E(n\<1\>)=n\<1\>\oplus\frac{n^\ell-n}{\ell}\cdot \ell\<1\>=n^\ell\<1\>.
\]
If $\ell=2$, then
\begin{align*}
\N_F^E(n\<1\>)&=n\<1\>\oplus\frac{n^2-n}{2}\tr_F^E\<1\>\\
&=n\<1\>\oplus\frac{n^2-n}{2}(\<1\>\oplus\<\alpha\>)\\
&=(n^2-1)\<1\>\oplus\<\alpha^{\frac{n^2-n}{2}}\>
\end{align*}

We now analyze the effect of $\N_F^E$ on terms of the form $(n-1)\<1\>\oplus\<\beta\>$, assuming that $\ell>2$. By Lemma \ref{lemma:TR}, we have
\begin{align*}
\N_F^E((n-1)\<1\>\oplus\<\beta\>)
  &=\N_F^E((n-1)\<1\>)\oplus\N_F^E(\<\beta\>)\oplus\tr_F^E\bigg(\sum_{i=1}^{\ell-1}\frac{{\binom{\ell}{i}}}{\ell}(n-1)^{\ell-i}\<\beta^i\>\bigg)\\
  &=(n-1)^\ell\<1\>\oplus\<\alpha\>\oplus\sum_{i=0}^{\ell-1}{\binom{\ell}{i}}(n-1)^{i}\<\alpha^i\>\\
  &=\sum_{i=0}^{\ell}{\binom{\ell}{ i}}(n-1)^{\ell-i}\<\alpha^i\>
\end{align*}

The isometry class of this quadratic form is determined by its dimension, $n^\ell$, and determinant. The latter quantity is
\[
\sum_{i \text{ odd}}^{\ell}{\binom{\ell}{i}}(n-1)^{\ell-i}\equiv\begin{cases}
\sum_{i \text{ odd}}^{\ell}{\binom{\ell}{i}}\equiv2^{\ell-1}\equiv0&\text{if $n\equiv 0\pmod 2$,}\\
{\binom{\ell}{\ell}}(n-1)^0\equiv1&\text{if $n\equiv 1\pmod 2$.}
\end{cases}
\]
We conclude that when $\ell$ is odd, the norm map is given by
\begin{align*}
\N_F^E(n\<1\>)&=n^\ell\<1\>,\\
\N_F^E((n-1)\<1\>\oplus\<\beta\>&=(n^\ell-1)\<1\>\oplus\<\alpha^n\>.
\end{align*}

It remains to determine $\N_F^E((n-1)\<1\>\oplus\<\beta\>)$ when $\ell=2$. In this case, we have
\begin{align*}
\N_F^E((n-1)\<1\>\oplus\<\beta\>)&=\N_F^E((n-1)\<1\>)\oplus\N_F^E(\<\beta\>)\oplus\tr_F^E(g_e((n-1)\<1\>,\<\beta\>))\\
&=((n-1)^2-1)\<1\>\oplus\<\alpha^{\frac{(n-1)^2-n+1}{2}}\>\oplus\<\alpha\>\oplus\tr_F^E((n-1)\<\beta\>)\\
&=((n-1)^2-1)\<1\>\oplus\<\alpha^{\frac{(n-1)^2-n+1}{2}}\>\oplus\<\alpha\>\oplus2(n-1)\<1\>\\
&=(n^2-1)\<1\>\oplus\<\alpha^{\frac{n^2-3n}{2}}\>
\end{align*}
This covers the final case and concludes the proof.
\end{proof}

\begin{theorem}[Norms for extensions of finite fields]\label{thm:main}
Let $F=\FF_q\subseteq \FF_{q^m}=E$ and fix generators $\alpha\in F^\times$, $\beta\in E^\times$. Then
\begin{align*}
\N_F^E(n\<1\>)&=\begin{cases}
n^m\<1\>&\text{$m$ odd,}\\
(n^m-1)\<1\>\oplus\<\alpha^{\frac{n^2-n}{2}}\>&\text{$m=2$,}\\
(n^m-1)\<1\>\oplus\<\alpha^{\frac{n^3-n^2}{2}}\>&\text{$m>2$ even,}
\end{cases}\\
\N_F^E((n-1)\<1\>\oplus\<\beta\>)&=\begin{cases}
(n^m-1)\<1\>\oplus\<\alpha^n\>&\text{$m$ odd,}\\
(n^m-1)\<1\>\oplus\<\alpha^{\frac{n^2-3n}{2}}\>&\text{$m=2$,}\\
(n^m-1)\<1\>\oplus\<\alpha^{\frac{n^3-3n^2}{2}}\>&\text{$m>2$ even.}
\end{cases}
\end{align*}
\end{theorem}

\begin{proof}
We use functoriality and Theorem \ref{thm:prime} to inductively determine the norms. Firstly, we test the composition of norm for two odd prime extensions. Let $L\subseteq F\subseteq E$ be finite fields of odd prime $s,t$ extension. Take some generator $\alpha\in L^\times$, $\gamma\in F^\times$, and $\beta\in E^\times$. Then we know that
\begin{align*}
\N_L^F\circ\N_F^E(n\<1\>)&=\N_L^F(n^t\<1\>)\\
&=n^{st}\<1\>\\
\N_L^F\circ\N_F^E((n-1)\<1\>\oplus\<\beta\>)&=\N_L^F((n^t-1)\<1\>\oplus\<\gamma^n\>)\\
&=\begin{cases}(n^{st}-1)\<1\>\oplus\<\alpha^{n^t}\>&n \text{ odd}\\
n^{st}\<1\>&n \text{ even}\end{cases}\\
&=(n^{st}-1)\<1\>\oplus\<\alpha^n\>.
\end{align*}
Inductively, we know that for any odd $m$ extension,
\begin{align*}
\N_F^E(n\<1\>)&=n^m\<1\>\\
\N_F^E((n-1)\<1\>\oplus\<\beta\>)&=(n^m-1)\<1\>\oplus\<\alpha^n\>.
\end{align*}

Similarly, we discuss the cases when $s=2$ and $t$ is odd. The computation gives us the result that for such extension,
\begin{align*}
\N_L^F\circ\N_F^E(n\<1\>)&=\N_L^F(n^t\<1\>)\\
&=(n^{st}-1)\<1\>\oplus\<\alpha^{\frac{(n^t)^2-n^t}{2}}\>\\
&=\begin{cases}(n^{st}-1)\<1\>\oplus\<\alpha\>&n\equiv 3\pmod 4\\
n^{st}\<1\>&n\equiv 0,1,2\pmod 4\end{cases}\\
&=(n^{st}-1)\<1\>\oplus\<\alpha^{\frac{n^3-n^2}{2}}\>\\
\N_L^F\circ\N_F^E((n-1)\<1\>\oplus\<\beta\>)&=\N_L^F((n^t-1)\<1\>\oplus\<\gamma^n\>)\\
&=\begin{cases}\N_L^F((n^t-1)\<1\>\oplus\<\alpha\>)&n\equiv 1,3\pmod 4\\
\N_L^F(n^t\<1\>)&n\equiv 0,2\pmod 4\end{cases}\\
&=\begin{cases}(n^{st}-1)\<1\>\oplus\<\alpha\>&n\equiv 1\pmod 4\\
n^{st}\<1\>&n\equiv 0,2,3\pmod 4\end{cases}\\
&=(n^{st}-1)\<1\>\oplus\<\alpha^{\frac{n^3-3n^2}{2}}\>
\end{align*}

For $s=t=2$, we again compute that 
\begin{align*}
\N_L^F\circ\N_F^E(n\<1\>)&=\N_L^F((n^t-1)\<1\>\oplus\<\gamma^{\frac{n^2-n}{2}}\>)\\
&=\begin{cases}\N_L^F((n^t-1)\<1\>\oplus\<\gamma\>)&n\equiv 2,3\pmod 4\\
\N_L^F(n^t\<1\>)&n\equiv 0,1\pmod 4\end{cases}\\
&=\begin{cases}(n^{st}-1)\<1\>\oplus\<\alpha\>&n\equiv 3\pmod 4\\
n^{st}\<1\>&n\equiv 0,1,2\pmod 4\end{cases}\\
&=(n^{st}-1)\<1\>\oplus\<\alpha^{\frac{n^3-n^2}{2}}\>\\
\N_L^F\circ\N_F^E((n-1)\<1\>\oplus\<\beta\>)&=\N_L^F((n^t-1)\<1\>\oplus\<\gamma^{\frac{n^2-3n}{2}}\>)\\
&=\begin{cases}\N_L^F((n^t-1)\<1\>\oplus\<\gamma\>)&n\equiv 1,2\pmod 4\\
\N_L^F(n^t\<1\>)&n\equiv 0,3\pmod 4\end{cases}\\
&=\begin{cases}(n^{st}-1)\<1\>\oplus\<\alpha\>&n\equiv 1\pmod 4\\
n^{st}\<1\>&n\equiv 0,2,3\pmod 4\end{cases}\\
&=(n^{st}-1)\<1\>\oplus\<\alpha^{\frac{n^3-3n^2}{2}}\>
\end{align*}

Therefore, for a $m$ extension where $m=4$ or $m=2t$, $t$ odd, the following result holds.
\begin{align*}
\N_F^E(n\<1\>)&=(n^{m}-1)\<1\>\oplus\<\alpha^{\frac{n^3-n^2}{2}}\>\\
\N_F^E((n-1)\<1\>\oplus\<\beta\>)&=(n^{m}-1)\<1\>\oplus\<\alpha^{\frac{n^3-3n^2}{2}}\>
\end{align*}

Moreover, when compositing such $m,k$ extension,
\begin{align*}
\N_L^F\circ\N_F^E(n\<1\>)&=\N_L^F((n^k-1)\<1\>\oplus\<\gamma^{\frac{n^3-n^2}{2}}\>)\\
&=\begin{cases}\N_L^F((n^k-1)\<1\>\oplus\<\gamma\>)&n\equiv 3\pmod 4\\
\N_L^F(n^k\<1\>)&n\equiv 0,1,2\pmod 4\end{cases}\\
&=\begin{cases}(n^{st}-1)\<1\>\oplus\<\alpha^{\frac{n^{3k}-3n^k}{2}}\>&n\equiv 3\pmod 4\\
(n^{st}-1)\<1\>\oplus\<\alpha^{\frac{n^{3k}-n^k}{2}}\>&n\equiv 0,1,2\pmod 4\end{cases}\\
&=(n^{st}-1)\<1\>\oplus\<\alpha^{\frac{n^3-n^2}{2}}\>\\
\N_L^F\circ\N_F^E((n-1)\<1\>\oplus\<\beta\>)&=\N_L^F((n^t-1)\<1\>\oplus\<\gamma^{\frac{n^2-3n}{2}}\>)\\
&=\begin{cases}\N_L^F((n^t-1)\<1\>\oplus\<\gamma\>)&n\equiv 1\pmod 4\\
\N_L^F(n^t\<1\>)&n\equiv 0,2,3\pmod 4\end{cases}\\
&=\begin{cases}(n^{st}-1)\<1\>\oplus\<\alpha^{\frac{n^{3k}-3n^k}{2}}\>&n\equiv 1\pmod 4\\
(n^{st}-1)\<1\>\oplus\<\alpha^{\frac{n^{3k}-n^k}{2}}\>&n\equiv 0,2,3\pmod 4\end{cases}\\
&=(n^{st}-1)\<1\>\oplus\<\alpha^{\frac{n^3-3n^2}{2}}\>
\end{align*}

Inductively, for any even $m$ extension other then 2, the following result holds.
\begin{align*}
\N_F^E(n\<1\>)&=(n^{m}-1)\<1\>\oplus\<\alpha^{\frac{n^3-n^2}{2}}\>\\
\N_F^E((n-1)\<1\>\oplus\<\beta\>)&=(n^{m}-1)\<1\>\oplus\<\alpha^{\frac{n^3-3n^2}{2}}\>
\end{align*}
This proves the theorem.
\end{proof}



\bibliographystyle{abbrv}
\bibliography{references}


\end{document}